\documentclass[a4paper,dvips,11pt]{article}
\usepackage[latin1]{inputenc}
\usepackage{amsmath,amsthm,amsfonts,amssymb}
\usepackage{enumerate}
\usepackage{graphicx,psfrag}
\usepackage[ps2pdf,colorlinks]{hyperref}
\usepackage{pst-node,pstricks-add,pst-text,pst-3d,pst-tree}

\newlength{\hdelta}
\newlength{\vdelta}
\theoremstyle{plain}
\newtheorem{theorem}{Theorem}
\newtheorem{lemma}{Lemma}
\newtheorem{proposition}{Proposition}

\theoremstyle{definition}
\newtheorem{remark}{Remark}

\newcommand{\R}{\mathbb{R}}
\newcommand{\mB}{\mathcal{B}}
\newcommand{\mF}{\mathcal{F}}
\newcommand{\E}{E}
\newcommand{\pr}{P}
\newcommand{\lest}{\le_{\rm{st}}}
\newcommand{\leicx}{\le_{\rm{icx}}}

\newcommand{\mailto}[1]{\href{mailto:#1}{\nolinkurl{#1}}}

\begin{document}

\title{Stochastic order characterization of \\ uniform integrability and tightness}

\author{
 Lasse Leskelä\thanks{
 Postal address:
 Department of Mathematics and Statistics,
 University of Jyväskylä,
 PO Box 35,
 40014 University of Jyväskylä, Finland
 Tel: +358 14 260 2728.
 URL: \url{http://www.iki.fi/lsl/} \quad
 Email: \protect\mailto{lasse.leskela@iki.fi}}
 \and
 Matti Vihola\thanks{
 Postal address:
 Department of Mathematics and Statistics,
 University of Jyväskylä,
 PO Box 35,
 40014 University of Jyväskylä, Finland
 Tel: +358 14 260 2713.
 URL: \url{http://iki.fi/mvihola/} \quad
 Email: \protect\mailto{mvihola@iki.fi}}
}

\date{\today}

\maketitle

\begin{abstract}
We show that a family of random variables is uniformly integrable if and only if it is
stochastically bounded in the increasing convex order by an integrable random variable. This
result is complemented by proving analogous statements for the strong stochastic order and for
power-integrable dominating random variables. Especially, we show that whenever a family of random
variables is stochastically bounded by a $p$-integrable random variable for some $p>1$, there is
no distinction between the strong order and the increasing convex order. These results also yield
new characterizations of relative compactness in Wasserstein and Prohorov metrics.
\end{abstract}


\vspace{1ex}


\vspace{1ex}

\section{Introduction}

Let $\{X_n\}$ be a sequence of random variables such that $X_n \to X$ almost surely. Lebesgue's
classical dominated convergence theorem implies that
\begin{equation}
 \label{eq:ConvergenceL1}
 \E |X_n - X| \to 0
\end{equation}
if there exists an integrable random variable $Y$ such that
\[
 |X_n| \le Y \quad \text{almost surely for all $n$}.
\]
A probabilistic version (e.g.\ \cite[Thm.~9.1]{Thorisson_2000}) of the above condition is to
require that
\begin{equation}
 \label{eq:BoundAS}
 |X_n| \lest Y \quad \text{for all $n$},
\end{equation}
where $\lest$ denotes the strong (a.k.a.\ usual) stochastic order
\cite{Leskela_2010,Muller_Stoyan_2002,Shaked_Shanthikumar_2007}. It is well known that in general
\eqref{eq:BoundAS} is not necessary for \eqref{eq:ConvergenceL1}. We will show that a sharp
characterization can be obtained when the strong stochastic order in~\eqref{eq:BoundAS} is
replaced by a weaker one. Namely, we will show that~\eqref{eq:ConvergenceL1} holds if and only if
there exists an integrable random variable $Y$ such that
\begin{equation}
 \label{eq:BoundICX}
 |X_n| \leicx Y \quad \text{for all $n$},
\end{equation}
where $\leicx$ denotes the increasing convex stochastic order
\cite{Muller_Stoyan_2002,Shaked_Shanthikumar_2007}.

More generally, our first main result, Theorem~\ref{the:ICXBound}, shows that a family of random
variables $\{X_n\}$ is uniformly integrable if and only if~\eqref{eq:BoundICX} holds. For almost
surely convergent random sequences, this result yields the equivalence of~\eqref{eq:ConvergenceL1}
and~\eqref{eq:BoundICX} (e.g. \cite[Prop 4.12]{kallenberg2002}). From the analysis point of view,
the characterization of uniform integrability in terms of the increasing convex order can be seen
as a new way to represent domination in Lebesgue's dominated convergence theorem in the natural
manner. What makes this result important in probability is that convex and increasing convex
orders are intimately connected with the existence of martingales and submartingales with given
marginals (e.g.\ Kellerer~\cite{Kellerer_1972}; Hirsch and Yor \cite{Hirsch_Yor_2010} and
references therein).

The second main result, Theorem~\ref{the:IcxImpliesSt}, shows that when studying whether a family
of random variables is stochastically bounded by a random variable in $L^p$ for some $p>1$, there
is no need to distinguish between the strong order and the increasing convex order. This somewhat
surprising result, which is a consequence of a Hardy--Littlewood maximal inequality, may open new
ways to establishing strong stochastic bounds using tools of convex analysis.

The main results are complemented by Proposition~\ref{the:BoundStrongFinite}, which states that
$\{X_n\}$ is tight if and only if~\eqref{eq:BoundAS} holds for some (almost surely finite) random
variable $Y$. This simple result is probably well known, because `bounded in probability' and
`stochastically bounded' are commonly used as synonyms for `tight'. We have formulated it here
explicitly in order to complete the big picture on various stochastic boundedness relationships in
Theorem~\ref{the:summary}. The implication diagram in Figure~\ref{fig:key} summarizes the findings
and illustrates how these concepts are related to relative compactness with respect to Wasserstein
and Prohorov metrics. Figure~\ref{fig:key} provides a new unified view to earlier studies on
lattice properties of increasing convex and strong stochastic orders (Kertz and
Rösler~\cite{Kertz_Rosler_2000}; Müller and Scarsini~\cite{Muller_Scarsini_2006}).

The rest of this paper is organized as follows.
Section \ref{sec:Definitions} introduces definitions and notation.
Section~\ref{sec:UniformIntegrability} discusses uniform integrability and tightness of
positive random variables, and Section~\ref{sec:Lp} extends the analysis to power-integrable random variables.
Section~\ref{sec:Summary} summarizes the main results in a diagram (Figure~\ref{fig:key})
and presents counterexamples confirming the sharpness of the implications in the diagram.

\section{Definitions and notation}
\label{sec:Definitions}

\subsection{Uniform integrability and tightness}

In general, we shall assume that $\{X_\alpha\}$ is a family of random variables with values in the
$d$-dimensional Euclidean space $\R^d$, indexed by a parameter $\alpha$ taking values in an
arbitrary set. The random variables do not need to be defined on a common probability space;
instead, we assume that each $X_\alpha$ is a random variable defined on some probability space
$(\Omega_\alpha,\mF_\alpha,\pr_\alpha)$. To keep the notation light, we replace $\pr_\alpha$ by
$\pr$, and denote the expectation with respect to $\pr_\alpha$ by $\E$. In addition, we denote by
$\mu_\alpha = \pr \circ X_\alpha^{-1}$ the distribution of $X_\alpha$ defined on the Borel sets of
$\R^d$. When $d=1$, the distribution function of $X_\alpha$ is defined by $F_\alpha(t) =
\pr(X_\alpha \le t)$, and the quantile function by $F_\alpha^{-1}(u) = \inf\{t: F_\alpha(t) \ge
u\}$ for $u \in (0,1)$.

A family of random variables $\{X_\alpha\}$ is called:
\begin{itemize}
 \item \emph{uniformly integrable} if
 $\sup_\alpha \E |X_\alpha| 1(|X_\alpha| > t) \to 0$ as $t \to \infty$,
 \item \emph{uniformly $p$-integrable} if
 $\sup_\alpha \E |X_\alpha|^p 1(|X_\alpha| > t) \to 0$ as $t \to \infty$,
 \item \emph{tight} if $\sup_\alpha \pr( |X_\alpha| > t ) \to 0$ as $t \to \infty$, and
 \item \emph{bounded in $L^p$} if $\sup_\alpha \E |X_\alpha|^p < \infty$.
\end{itemize}

\begin{remark}
All results in this article remain valid also in a slightly more general setting where $\R^d$ is
replaced by a complete separable metric space $S$ satisfying the Heine--Borel property: a subset
of $S$ is compact if and only if it is bounded and closed. In this case we replace the norm of $x$
by $|x| = \rho(x,x_0)$, where $\rho$ is the metric and $x_0$ is an arbitrary reference point in
$S$. For concreteness, we prefer to formulate the results for $\R^d$-valued random variables.
\end{remark}

\subsection{Stochastic orders}

For real-valued random variables $X_1$ and $X_2$, we denote $X_1 \lest X_2$ and say that $X_1$ is
less than $X_2$ in the \emph{strong stochastic order}, or $X_1$ is \emph{st-bounded} by $X_2$, if
$\E \phi(X_1) \le \E \phi(X_2)$ for all increasing\footnote{The symbol $\R_+$ refers to positive
real numbers including zero. In general, we use the terms `positive', `increasing', and `less
than' as synonyms for `nonnegative', `nondecreasing', and `less or equal than'.} measurable
functions $\phi: \R \to \R_+$. It is well known
\cite{Leskela_2010,Muller_Stoyan_2002,Shaked_Shanthikumar_2007} that the following are equivalent
for any real-valued random variables $X_1$ and $X_2$ with distribution functions $F_1$ and $F_2$:
\begin{enumerate}[(i)]
  \item $X_1 \lest X_2$.
  \item $1-F_1(t) \le 1-F_2(t)$ for all $t \in \R$.
  \item $F_1^{-1}(u) \le F_2^{-1}(u)$ for all $u \in (0,1)$.
  \item There exist random variables $\hat X_1$ and $\hat X_2$ defined on a common probability space, distributed according to $X_1$ and $X_2$,
      such that $\hat X_1 \le \hat X_2$ almost surely.
\end{enumerate}

For positive random variables $X_1$ and $X_2$, we denote $X_1 \leicx X_2$ and say that $X_1$ is
less than $X_2$ in the \emph{increasing convex order}, or $X_1$ is \emph{icx-bounded} by $X_2$, if
$\E \phi(X) \le \E \phi(Y)$ for all increasing convex functions $\phi: \R_+ \to \R_+$. A
fundamental characterization (e.g. \cite[Thm.~1.5.7]{Muller_Stoyan_2002}) of increasing convex
stochastic orders states that
\[
 X_1 \leicx X_2 \quad \text{if and only if \quad $H_1(t) \le H_2(t)$ for all $t \ge 0$},
\]
where
\[
 H_i(t) = \E(X_i - t)_+ = \int_t^\infty (1-F_i(t)) \, dt
\]
denotes the \emph{integrated survival function} of $X_i$, and $x_+ = \max(x,0)$. Basic technical
facts about integral survival functions are given in Appendix~\ref{sec:ISF}.

\begin{remark}
Note that if $\E X_2 = \infty$, then any positive random variable is icx-bounded by $X_2$. Therefore
to say that a random variable $X_1$ is icx-bounded by $X_2$ truly makes sense only when $X_2$ is integrable, in which case also $X_1$ is integrable.
\end{remark}

\begin{remark}
The notion of increasing convex order can also be extended to random variables on the full real
line having an integrable right tail. Because the main results of this article concern norms of
random variables, for simplicity we prefer to restrict our terminology to positive random
variables, using `integrable' in place of `having an integrable right tail' etc.
\end{remark}

\section{Uniform integrability and tightness of positive random variables}
\label{sec:UniformIntegrability}

The following theorem, the first main result of this article, provides a sharp characterization of
uniform integrability as stochastic boundedness in the increasing convex order. We say that a
family of positive random variables $\{X_\alpha\}$ is \emph{icx-bounded} by a random variable $Y$
if $X_\alpha \leicx Y$ for all $\alpha$.

\begin{theorem}
\label{the:ICXBound}
The following are equivalent for any family of positive random variables $\{X_\alpha\}$ with
distribution functions $F_\alpha$:
\begin{enumerate}[(i)]
\item \label{pro:ICXBound1} $\{X_\alpha\}$ is icx-bounded by an integrable random variable.
\item \label{pro:ICXBound2} $\{X_\alpha\}$ is uniformly integrable.
\item \label{pro:ICXBound3} $\lim_{t \to \infty} \sup_\alpha \int_t^\infty (1-F_\alpha(u)) \, du =
0$.
\end{enumerate}
\end{theorem}

\begin{proof}
\eqref{pro:ICXBound1} $\implies$ \eqref{pro:ICXBound2}. Assume that $X_\alpha \leicx X$ for all
$\alpha$, where $X$ has a finite mean. Then $\E (X_\alpha-t)_+ \le \E (X-t)_+$ for all $t \ge 0$.
Note that
\[
 x 1(x > t) \le 2(x-t/2)_+
\]
for all $x$ and $t$. Therefore,
\[
 \E X_\alpha 1(X_\alpha > t) \le 2 \E (X_\alpha - t/2)_+,
\]
which implies that
\[
 \sup_\alpha \E X_\alpha 1(X_\alpha > t) \le 2 \E (X - t/2)_+.
\]
Because $X$ has a finite mean, we see by using dominated convergence that the right side above
tends to zero as $t$ grows, and $\{X_\alpha\}$ is hence uniformly integrable.

\eqref{pro:ICXBound2} $\implies$ \eqref{pro:ICXBound3}. Because
\[
 \int_t^\infty (1-F_\alpha(u)) \, du
 = \E (X_\alpha - t)_+
 \le \E X_\alpha 1(X_\alpha > t)
\]
for all $t \ge 0$ and all $\alpha$, the uniform integrability of $\{X_\alpha\}$
implies~\eqref{pro:ICXBound3}.

\eqref{pro:ICXBound3} $\implies$ \eqref{pro:ICXBound1}. Define $H(t) = \sup_\alpha H_\alpha(t)$,
where $H_\alpha(t) = \E (X_\alpha - t)_+$ is the integrated survival function of $X_\alpha$. Note
that $H_\alpha$ is convex and $H_\alpha(t) + t \ge H_\alpha(0)$ for all $t \ge 0$ by
Lemma~\ref{the:IFS} in Appendix~\ref{sec:ISF}. As a consequence, the same properties are valid for
$H$. Further, $H(t) \to 0$ as $t \to \infty$ by~\eqref{pro:ICXBound3}. Therefore, by
Lemma~\ref{the:IFS} we conclude that $H$ is the integrated survival function of an integrable
positive random variable $X$. By the definition of $H$, it follows that for all $\alpha$,
\[
 \E (X_\alpha - t)_+ = H_\alpha(t) \le H(t) = \E (X-t)_+ \quad \text{for all $t \ge 0$},
\]
which is equivalent to $X_\alpha \leicx X$ \cite[Thm.~1.5.7]{Muller_Stoyan_2002}.
\end{proof}

\begin{remark}
\label{rem:UIAlt}
The role of increasing convex functions in characterizing uniform integrability is also visible in
the de la Vallée-Poussin theorem, which states that a family of random variables $\{X_\alpha\}$ in
$\R^d$ is uniformly integrable if and only if
\[
 \sup_\alpha \E \phi(|X_\alpha|) < \infty
\]
for some increasing convex function $\phi$ such that $\phi(t)/t \to \infty$ as $t \to \infty$ \cite[Prop.~A.2.2]{Ethier_Kurtz_1986}.
\end{remark}

The following simple result characterizes tightness as stochastic boundedness in the strong order.
We say that a family of random variables $\{X_\alpha\}$ is \emph{st-bounded} by a random variable
$Y$ if $X_\alpha \lest Y$ for all $\alpha$. A positive random variable $X$ (which by convention is
allowed to take on the value $\infty$) is called \emph{finite} if $\pr(X < \infty) = 1$.

\begin{proposition}
\label{the:BoundStrongFinite}
The following are equivalent for any family of positive random variables $\{X_\alpha\}$ with
distribution functions $F_\alpha$:
\begin{enumerate}[(i)]
\item \label{pro:BoundStrongFinite1} $\{X_\alpha\}$ is st-bounded by a finite random variable.
\item \label{pro:BoundStrongFinite2} $\{X_\alpha\}$ is tight.
\item \label{pro:BoundStrongFinite3} $\lim_{t \to \infty} \sup_\alpha (1- F_\alpha(t)) = 0$.
\end{enumerate}
\end{proposition}
\begin{proof}
\eqref{pro:BoundStrongFinite1}$\implies$\eqref{pro:BoundStrongFinite2}. Assume that $X_\alpha
\lest X$ for all $\alpha$, where $X$ is a finite random variable. Then
\[
 \sup_\alpha \pr(X_\alpha > t) \le \pr(X > t).
\]
Because the right side above tends to zero as $t \to \infty$, tightness follows.

\eqref{pro:BoundStrongFinite2}$\implies$\eqref{pro:BoundStrongFinite3}. Clear.

\eqref{pro:BoundStrongFinite3}$\implies$\eqref{pro:BoundStrongFinite1}. Let $F(t) = \inf_\alpha
F_\alpha(t)$. Then Lemma~\ref{the:CDF} in Appendix~\ref{sec:DF} shows that $F$ is a distribution
function of a finite positive random variable $X_*$. Because $1 - F_\alpha(t) \le 1 - F(t)$ for
all $t$ and all $\alpha$, it follows that $X_\alpha \lest X_*$ for all $\alpha$. \end{proof}

\begin{remark}
\label{rem:TightAlt}
An alternative equivalent characterization of tightness, analogous to Remark~\ref{rem:UIAlt}, is
the following: A family of random variables $\{X_\alpha\}$ in $\R^d$ is tight if and only if
\[
 \sup_\alpha \E \phi(|X_\alpha|) < \infty
\]
for some positive function $\phi$ such that $\phi(t) \to \infty$ as $t \to \infty$
\cite[Lem.~D.5.3]{Meyn_Tweedie_1993}.
\end{remark}

\section{Stochastic boundedness of power-integrable random variables}
\label{sec:Lp}

\subsection{Boundedness in strong vs.\ increasing convex order}

If a family of positive random variables is st-bounded by a $p$-integrable random variable for
some $p>0$, then the same is obviously true when `st-bounded' is replaced by `icx-bounded'. The
following result shows that, rather surprisingly, these two properties are equivalent when $p>1$.

\begin{theorem}
\label{the:IcxImpliesSt}
For any $p>1$, a family of positive random variables $\{X_\alpha\}$ is icx-bounded by a
$p$-integrable random variable if and only if $\{X_\alpha\}$ is st-bounded by a $p$-integrable
random variable.
\end{theorem}

To prove this result, we will apply Hardy--Littlewood maximal functions. Let $X$ be an integrable
random variable with distribution function $F$, and denote the quantile function of $X$ by
$F^{-1}(u) = \inf\{t: F(t) \ge u\}$ for $u \in (0,1)$. Let $U$ be a uniformly distributed random
variable in $(0,1)$ and define
\[
 X^* = MF^{-1}(U),
\]
where
\[
 MF^{-1}(u)
 = \sup_{u<w<1} (w-u)^{-1} \int_u^w F^{-1}(v) \, dv
 = (1-u)^{-1} \int_u^1 F^{-1}(v) \,
 dv
\]
is the Hardy--Littlewood maximal function of $F^{-1}$ \cite{Hardy_Littlewood_1930}. The function
$MF^{-1}$ is well-defined when $X$ is integrable~\cite{Dubins_Gilat_1978,Kertz_Rosler_2000}. We
call $X^*$ the \emph{Hardy--Littlewood maximal random variable} associated with $X$. A remarkable
result of Kertz and Rösler~\cite[Lem.~4.1]{Kertz_Rosler_2000} states that for any integrable
random variables $X$ and $Y$,
\begin{equation}
 \label{eq:MaximalFunctionEquivalence}
 X \leicx Y \quad \text{if and only if} \quad X^* \lest Y^*.
\end{equation}

\begin{lemma}
\label{the:MaximalRV}
$X \lest X^*$ for any integrable random variable $X$.
\end{lemma}
\begin{proof}
Observe that the quantile function of $X^*$ equals $MF^{-1}$. Further, the monotonicity of
$F^{-1}$ shows that $F^{-1}(u) \le MF^{-1}(u)$ for all $u \in (0,1)$. This pointwise ordering of
the quantile functions is a well-known necessary and sufficient condition for the strong ordering
of the corresponding random variables (e.g.\ \cite[Thm.~1.2.4]{Muller_Stoyan_2002}).
\end{proof}

\begin{proof}[Proof of Theorem~\ref{the:IcxImpliesSt}]
Assume that $X_\alpha \leicx Y$ for all $\alpha$, where $Y \in L^p$. Then $X_\alpha^* \lest Y^*$
for all $\alpha$ due to~\eqref{eq:MaximalFunctionEquivalence}, so by applying
Lemma~\ref{the:MaximalRV} we see that the family $\{X_\alpha\}$ is st-bounded by $Y^*$. It remains
to verify that $Y^*$ is in $L^p$. Note that the quantile function of $Y^*$ equals $MF_Y^{-1}$,
where $F_Y^{-1}$ is the quantile function of $Y$. A classic inequality of Hardy and
Littlewood~\cite{Hardy_Littlewood_1930} now implies that
\[
 \E (Y^*)^p
 = \int_0^1 \left( MF_Y^{-1}(u) \right)^p \, du
 \le c \int_0^1 \left( F_Y^{-1}(u) \right)^p \, du
 = c \E Y^p,
\]
where $c= \left( \frac{p}{p-1} \right)^p$.
\end{proof}

\subsection{Power-integrable strong bounds}

The following result shows how to reduce the study of strong boundedness by a $p$-integrable
random variable to strong boundedness by an integrable random variable. A corresponding statement
is not true for the increasing convex ordering, see Theorem~\ref{the:summary} in
Section~\ref{sec:Summary}.

\begin{proposition}
\label{the:BoundStrongIntegrable}
For any $p>0$ and any family of positive random variables $\{X_\alpha\}$ with distribution
functions $F_\alpha$, the following are equivalent:
\begin{enumerate}[(i)]
\item \label{pro:BoundStrongIntegrable1} $\{X_\alpha\}$ is st-bounded by a $p$-integrable random variable.
\item \label{pro:BoundStrongIntegrable2} $\{X_\alpha^p\}$ is st-bounded by an integrable random variable.
\item \label{pro:BoundStrongIntegrable3} $\int_0^\infty \sup_\alpha (1-F_\alpha(t^{1/p})) \, dt < \infty$.
\end{enumerate}
\end{proposition}
\begin{proof}
\eqref{pro:BoundStrongIntegrable1}$\implies$\eqref{pro:BoundStrongIntegrable2}. Assume that
$X_\alpha \lest Y$ for all $\alpha$, for some $p$-integrable random variable $Y$. Because the
random variables $X_\alpha$ are positive, it follows that $Y \ge 0$ almost surely. Because $x
\mapsto x^p$ is increasing on $\R_+$, it follows that $X^p_\alpha \lest Y^p$ for all $\alpha$.

\eqref{pro:BoundStrongIntegrable2}$\implies$\eqref{pro:BoundStrongIntegrable1}. The same argument
as above works also in this direction, because $x \mapsto x^{1/p}$ is increasing.

\eqref{pro:BoundStrongIntegrable1}$\implies$\eqref{pro:BoundStrongIntegrable3}. Assume that
$X_\alpha \lest Y$ for all $\alpha$, for some $p$-integrable random variable $Y \ge 0$. Define
$F(t) = \inf_\alpha F_\alpha(t)$. Proposition~\ref{the:BoundStrongFinite} implies that $\lim_{t
\to \infty} F(t) = 1$, so by Lemma~\ref{the:CDF} in Appendix~\ref{sec:DF}, we see that $F$ is a
distribution function of a finite positive random variable $Z$. Because $Z \lest X_\alpha$ for all
$\alpha$, it follows that $Z \lest Y$. The monotonicity of $x \mapsto x^p$ further implies that
$\E Z^p \le \E Y^p < \infty$. Therefore,
\[
 \int_0^\infty \sup_\alpha (1-F_\alpha(t^{1/p})) \, dt
 = \int_0^\infty (1-F(t^{1/p})) \, dt
 = \E Z^p < \infty.
\]

\eqref{pro:BoundStrongIntegrable3}$\implies$\eqref{pro:BoundStrongIntegrable1}. Define $F(t) =
\inf_\alpha F_\alpha(t)$, and observe that $\lim_{t \to \infty} F(t) = 1$
by~\eqref{pro:BoundStrongIntegrable3}. Again, by Lemma~\ref{the:CDF}, we see that $F$ is the
distribution function of some finite positive random variable $Z$, and $X_\alpha \lest Z$ for all
$\alpha$. Further, the formula
\[
 \E Z^p = \int_0^\infty (1-F(t^{1/p})) \, dt
\]
shows that $Z$ is $p$-integrable. \end{proof}

We conclude this section by a simple corollary of the above result.

\begin{proposition}
\label{the:PowerBoundedImpliesStBounded}
Let $0<p<q$. If a family of positive random variables $\{X_\alpha\}$ is bounded in $L^q$, then it
is st-bounded by a $p$-integrable random variable.
\end{proposition}
\begin{proof}
Markov's inequality implies that
\[
 1 - F_\alpha(t^{1/p}) = \pr( |X_\alpha|^{q} > t^{q/p} ) \le t^{-q/p} \E |X_\alpha|^q
\]
for all $\alpha$. Therefore,
\[
 \sup_\alpha (1 - F_\alpha(t^{1/p})) \le \max(1, c t^{-q/p}),
\]
where $c = \sup_\alpha \E |X_\alpha|^q$. Because the right side above is integrable, the claim
follows by Proposition~\ref{the:BoundStrongIntegrable}. \end{proof}

\section{Summary}
\label{sec:Summary}

This section concludes the paper with a diagram (Figure~\ref{fig:key}) which summarizes the relationships
between the various stochastic boundedness properties. To illustrate how
these concepts relate to compactness properties of probability measures, let us recall the definitions
of Prohorov and Wasserstein metrics.

The \emph{Prohorov metric} on the space $M$ of probability measures on $\R^d$ is defined by
\[
 d_P(\mu,\nu)
 = \inf \left\{ \epsilon > 0: \mu(B) \le \nu(B^\epsilon) + \epsilon
 \ \text{and} \
 \nu(B) \le \mu(B^\epsilon) + \epsilon
 \ \text{for all} \ B \in \mB
 \right\},
\]
where $B^\epsilon = \{x \in \R^d: |x-b| < \epsilon \ \text{for some} \ b \in B\}$ denotes the $\epsilon$-neighborhood
of $B$, and $\mB$ denotes the Borel sets of $\R^d$. The space $(M,d_P)$ is a complete separable metric space, and convergence
with respect to $d_P$ corresponds to convergence in distribution (e.g.\ \cite[Thm.~3.1.7]{Ethier_Kurtz_1986}).

For $p \ge 1$, denote by $M_p$ the space of probability measures on $\R^d$ with a finite $p$-th moment.
The $\emph{$p$-Wasserstein metric}$ on $M_p$ is defined by
\[
 d_{W,p}(\mu,\nu) = \left( \inf_{\gamma \in K(\mu,\nu)} \int_{\R^d \times \R^d} |x-y|^p \, \gamma(dx,dy) \right)^{1/p},
\]
where $K(\mu,\nu)$ is the set probability measures on $\R^d \times \R^d$ with first marginal $\mu$
and second marginal $\nu$. The space $(M_p,d_{W,p})$ is a complete separable metric space, and a
sequence converges with respect to $d_{W,p}$ if and only if it is uniformly $p$-integrable and
converges in distribution \cite[Prop.~7.1.5]{Ambrosio_Gigli_Savare_2008}. See Rachev and
Rüschendorf~\cite{Rachev_Ruschendorf_1998_I} for a detailed account of dual characterizations of
$d_{W,p}$, and Gibbs and Su~\cite{Gibbs_Su_2002} for a nice survey of different probability
metrics and their relationships.

\begin{theorem}
\label{the:summary} For any $0<p<1<q<\infty$ and for any family of random variables $\{X_\alpha\}$ in $\R^d$ with probability
distributions $\mu_\alpha = \pr \circ X_\alpha^{-1}$, the
implications in Figure~\ref{fig:key} are valid. In general, no other implications hold.
\end{theorem}

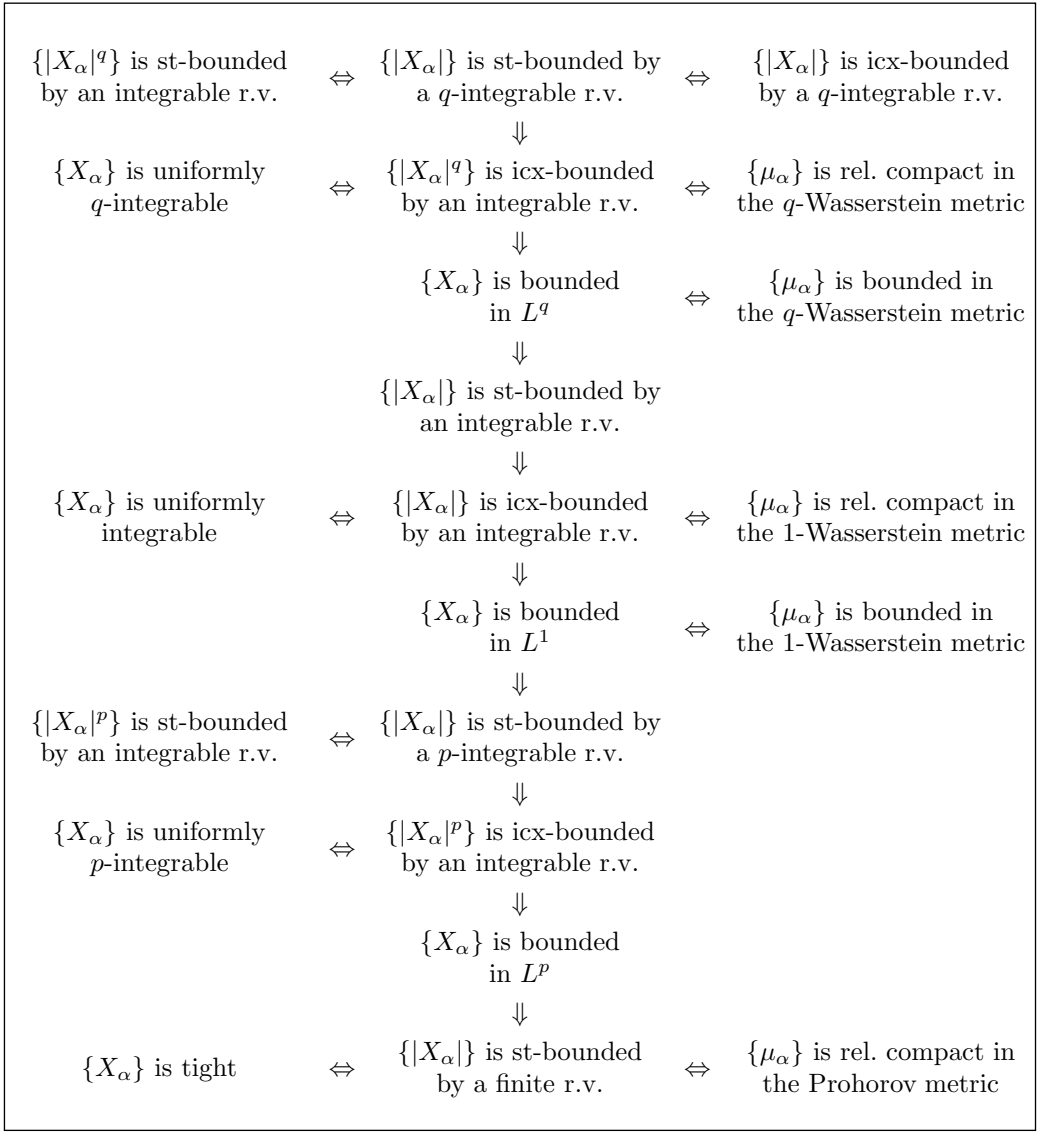
\begin{figure}[h!]
  \begin{center}
  \small
  \psset{unit=.73cm}  % Good for arXiv
  %\psset{unit=.64cm} % Good for Springer
  \begin{pspicture}(-9,-1)(+9,19.5)
  \psframe[linewidth=.5pt,framearc=0.00](-9.3,-1)(+9.3,19.5)

  \rput[B](0,18){\begin{tabular}{c} $\{ |X_\alpha| \}$ is st-bounded by \\ a $q$-integrable r.v. \end{tabular}}
  \rput[B](0,16){\begin{tabular}{c} $\{ |X_\alpha|^q \}$ is icx-bounded \\ by an integrable r.v. \end{tabular}}
  \rput[B](0,14){\begin{tabular}{c} $\{X_\alpha\}$ is bounded \\ in $L^q$ \end{tabular}}
  \rput[B](0,12){\begin{tabular}{c} $\{|X_\alpha|\}$ is st-bounded by \\ an integrable r.v. \end{tabular}}
  \rput[B](0,10){\begin{tabular}{c} $\{ |X_\alpha| \}$ is icx-bounded \\ by an integrable r.v. \end{tabular}}
  \rput[B](0,08){\begin{tabular}{c} $\{ X_\alpha \}$ is bounded \\ in $L^1$ \end{tabular}}
  \rput[B](0,06){\begin{tabular}{c} $\{|X_\alpha|\}$ is st-bounded by \\ a $p$-integrable r.v. \end{tabular}}
  \rput[B](0,04){\begin{tabular}{c} $\{ |X_\alpha|^p \}$ is icx-bounded \\ by an integrable r.v. \end{tabular}}
  \rput[B](0,02){\begin{tabular}{c} $\{X_\alpha\}$ is bounded \\ in $L^p$ \end{tabular}}
  \rput[B](0,00){\begin{tabular}{c} $\{|X_\alpha|\}$ is st-bounded \\ by a finite r.v. \end{tabular}}

  \rput[B](-6.5,18){\begin{tabular}{c} $\{ |X_\alpha|^q \}$ is st-bounded \\ by an integrable r.v. \end{tabular}}
  \rput[B](+6.5,18){\begin{tabular}{c} $\{ |X_\alpha| \}$ is icx-bounded \\ by a $q$-integrable r.v. \end{tabular}}
  \rput[B](-3.2,18){$\Leftrightarrow$}
  \rput[B](+3.2,18){$\Leftrightarrow$}

  \rput[B](-6.5,16){\begin{tabular}{c} $\{ X_\alpha \}$ is uniformly \\ $q$-integrable \end{tabular}}
  \rput[B](-3.2,16){$\Leftrightarrow$}
  \rput[B](+6.5,16){\begin{tabular}{c} $\{ \mu_\alpha \}$ is rel.\ compact in \\ the $q$-Wasserstein metric \end{tabular}}
  \rput[B](+3.2,16){$\Leftrightarrow$}

  \rput[B](+6.5,14){\begin{tabular}{c} $\{ \mu_\alpha \}$ is bounded in \\ the $q$-Wasserstein metric \end{tabular}}
  \rput[B](+3.2,14){$\Leftrightarrow$}

  \rput[B](-6.5,10){\begin{tabular}{c} $\{ X_\alpha \}$ is uniformly \\ integrable \end{tabular}}
  \rput[B](-3.2,10){$\Leftrightarrow$}
  \rput[B](+6.5,10){\begin{tabular}{c} $\{ \mu_\alpha \}$ is rel.\ compact in \\ the 1-Wasserstein metric \end{tabular}}
  \rput[B](+3.2,10){$\Leftrightarrow$}

  \rput[B](+6.5,08){\begin{tabular}{c} $\{ \mu_\alpha \}$ is bounded in \\ the $1$-Wasserstein metric \end{tabular}}
  \rput[B](+3.2,08){$\Leftrightarrow$}

  \rput[B](-6.5,06){\begin{tabular}{c} $\{ |X_\alpha|^p \}$ is st-bounded \\ by an integrable r.v. \end{tabular}}
  \rput[B](-3.2,06){$\Leftrightarrow$}

  \rput[B](-6.5,04){\begin{tabular}{c} $\{ X_\alpha \}$ is uniformly \\ $p$-integrable \end{tabular}}
  \rput[B](-3.2,04){$\Leftrightarrow$}

  \rput[B](+6.5,00){\begin{tabular}{c} $\{ \mu_\alpha \}$ is rel.\ compact in \\ the Prohorov metric \end{tabular}}
  \rput[B](+3.2,00){$\Leftrightarrow$}

  \rput[B](-6.5,00){\begin{tabular}{c} $\{ X_\alpha \}$ is tight \end{tabular}}
  \rput[B](-3.2,00){$\Leftrightarrow$}

  \multirput[B](0,17)(0,-2){9}{$\Downarrow$}
  \end{pspicture}
  \end{center}
  \caption{\label{fig:key} Stochastic boundedness relationships.}
\end{figure}

\begin{proof}
Let us start by verifying the equivalences on the left of Figure~\ref{fig:key}, starting from top.
The first and fourth and due to Proposition~\ref{the:BoundStrongIntegrable}. The second, third,
and fifth are due to Theorem~\ref{the:ICXBound}. The sixth is due to
Proposition~\ref{the:BoundStrongFinite}.

Let us next verify the equivalences on the right, again starting from the top. The first is due to
Theorem~\ref{the:IcxImpliesSt}. By applying Ambrosio, Gigli, and Savaré
\cite[Prop.~7.1.5]{Ambrosio_Gigli_Savare_2008}, we know that $\{\mu_\alpha\}$ is relatively
compact in the $q$-Wasserstein metric if and only if $\{X_\alpha\}$ is uniformly $q$-integrable
and tight. Because closed bounded sets are compact in $\R^d$, it follows that uniform
$q$-integrability implies tightness. Therefore, the second equivalence on the right holds. The
third equivalence follows by observing that the $L^q$-norm of $X_\alpha$ is equal to the
$q$-Wasserstein distance between $\mu_\alpha$ and the Dirac measure at zero. The fourth and fifth
equivalences are verified by similar reasoning. For the sixth, it suffices to recall that
tightness and relative compactness in the Prohorov metric are equivalent by Prohorov's classic
theorem (e.g.~\cite[Thm 16.3]{kallenberg2002}).

The downward implications are more or less immediate. Starting from top, the first, fourth, and
seventh are trivial, because strong ordering implies increasing convex ordering. The second,
fifth, and eighth are trivial as well, because increasing convex ordering implies the ordering of
the means. The third, sixth, and ninth are due to
Proposition~\ref{the:PowerBoundedImpliesStBounded}.

To complete the picture, we will now construct families of random variables which show that none of the one-way
implications in Figure~\ref{fig:key} can be reversed in general. Let $U$ be a uniformly
distributed random variable in $(0,1)$, and let $\phi_n$ and $\psi_n$, $n \ge 2$, be random
variables such that
\[
 \phi_n = \begin{cases}
   \text{$n$ with probability $n^{-1}$}, \\
   \text{$0$ else},
 \end{cases}
 \quad
 \psi_n = \begin{cases}
   \text{$n$ with probability $(n \log n)^{-1}$}, \\
   \text{$0$ else}.
 \end{cases}
\]
Then for any $0 < p < 1 < q$:
\begin{itemize}
  \item $\{e^{1/U}\}$ is st-bounded by a finite r.v.\ but not bounded in $L^p$.
  \item $\{\phi_n^{1/p}\}$ is bounded in $L^p$ but not uniformly $p$-integrable.
  \item $\{\psi_n^{1/p}\}$ is uniformly $p$-integrable but not st-bounded by a r.v.\ in $L^p$.
  \item $\{U^{-1}\}$ is st-bounded by a r.v.\ in $L^p$ but not bounded in $L^1$. \phantom{$\phi_n^{1/p}$}
  \item $\{\phi_n\}$ is bounded in $L^1$ but not uniformly integrable. \phantom{$\phi_n^{1/p}$}
  \item $\{\psi_n\}$ is uniformly integrable but not st-bounded by an integrable r.v. \phantom{$\phi_n^{1/p}$}
  \item $\{U^{-1/q}\}$ is st-bounded by an integrable r.v.\ but not bounded in $L^q$. \phantom{$\phi_n^{1/p}$}
  \item $\{\phi_n^{1/q}\}$ is bounded in $L^q$ but not uniformly $q$-integrable. \phantom{$\phi_n^{1/p}$}
  \item $\{\psi_n^{1/q}\}$ is uniformly $q$-integrable but not st-bounded by a r.v.\ in $L^q$. \phantom{$\phi_n^{1/p}$}
\end{itemize}
The straightforward computations required for verifying the above claims are left to the reader.
\end{proof}

\appendix

\section{Distribution functions}
\label{sec:DF}

If $X$ is a finite positive random variable, then its distribution function $F(t) = \pr(X \le t)$
is increasing, right-continuous, and satisfies $F(0) \ge 0$ and $\lim_{t \to \infty} F(t) = 1$.
Recall that any function $F$ with these properties can be realized as the distribution function of
the random variable $X = F^{-1}(U)$, where $F^{-1}(u) = \inf\{t \in \R_+: F(t) \ge u\}$ and $U$ is
a uniformly distributed random variable in $(0,1)$.

\begin{lemma}
\label{the:CDF}
For any family $\{F_\alpha\}$ of distribution functions on $\R_+$, the function $F(t) =
\inf_\alpha F_\alpha(t)$ is a distribution function on $\R_+$ if and only if
\begin{equation}
 \label{eq:Tight}
 \lim_{t \to \infty} \inf_\alpha F_\alpha(t) = 1.
\end{equation}
\end{lemma}
\begin{proof}
The necessity of~\eqref{eq:Tight} is obvious. To prove sufficiency, assume that a family of
distribution functions $\{F_\alpha\}$ on $\R_+$ satisfies~\eqref{eq:Tight}. Then it immediately
follows that $F$ is increasing, $F(0) \ge 0$, and $\lim_{t\to\infty} F(t) = 1$. Therefore, we only
need to verify that $F$ is right-continuous. Fix arbitrary $t \ge 0$ and $\epsilon > 0$, and
choose an index $\alpha$ such that $F_\alpha(t) < F(t) + \epsilon/2$. By the right-continuity of
$F_\alpha$ there exists a $\delta > 0$ such that $F_\alpha(t+h) - F_\alpha(t) < \epsilon/2$ for
all $h \in (0,\delta)$. As a consequence,
\[
 F(t+h) \le F_\alpha(t+h) < F_\alpha(t) + \epsilon/2 < F(t) + \epsilon
\]
for all $h \in (0,\delta)$. Because $F$ is increasing, this implies that $F$ is right-continuous.
\end{proof}

\section{Integrated survival functions}
\label{sec:ISF}

The integrated survival function of a positive integrable random variable $X$ with distribution
function $F$ is defined by
\[
 H(t) = \E (X-t)_+
 = \int_t^\infty (1-F(u)) \, du.
\]
The following result characterizes the family of integrated survival functions generated by
positive integrable random variables. The proof is similar to an analogous characterization for
random variables on the full real line (Müller and Stoyan~\cite[Thm.~1.5.10]{Muller_Stoyan_2002}).
Because our formulation of the result is slightly different, we include the proof here for the
reader's convenience.

\begin{lemma}
\label{the:IFS}
A function $H: \R_+ \to \R_+$ is the integrated survival function of a positive integrable random
variable if and only if
\begin{enumerate}[(i)]
\item \label{pro:IFSConvex}   $H$ is convex,
\item \label{pro:IFSLimit}    $\lim_{t \to \infty} H(t) = 0$,
\item \label{pro:IFSBoundary} $H(t) + t \ge H(0)$ for all $t \ge 0$.
\end{enumerate}
\end{lemma}
\begin{proof}
Assume first that $H$ is the integrated survival function of a positive integrable random variable
$X$. Then \eqref{pro:IFSConvex} follows immediately because $t \mapsto (x-t)_+$ is convex for all
$x$, and \eqref{pro:IFSLimit} follows by dominated convergence. To see the validity
of~\eqref{pro:IFSBoundary}, it suffices to observe that
\[
 H(t) + t = \E \max(X,t) \ge \E X = H(0).
\]

Assume next that a function $H: \R_+ \to \R_+$ satisfies
\eqref{pro:IFSConvex}--\eqref{pro:IFSBoundary}. The convexity of $H$ together with
\eqref{pro:IFSLimit} implies that $H$ is decreasing on $\R_+$. By applying \eqref{pro:IFSBoundary}
we see that $-t \le H(t) - H(0) \le 0$ for all $t$, which implies that $H$ is continuous at zero.
As a consequence, the right derivative $H'_+(t)$ of $H$ exists for all $t \ge 0$, and the function
$H'_+: \R_+ \to \R$ is right-continuous and increasing (e.g.\
\cite[Prop.~4.1.1]{Bertsekas_Nedic_Ozdaglar_2003}). Define $F(t) = 1 + H'_+(t)$. Then $F(0) \ge 0$
due to~\eqref{pro:IFSBoundary}, and $\lim_{t \to \infty} F(t) = 1$, because $H(t)$ decreases to
zero as $t \to \infty$. Because $F$ is right-continuous and increasing, we conclude that $F$ is
the distribution function of the random variable $X = F^{-1}(U)$, where $U$ is uniformly
distributed in $(0,1)$. To see that $H$ is the integrated survival function of $X$, we note by
changing the order of integration that for all $t$,
\[
 H(t)
 = - \int_t^\infty H_+'(u) \, du
 = \int_t^\infty (1-F(u)) \, du
 = \E (X-t)_+.
\]
Especially, $\E X = H(0)$ is finite, which shows that $X$ is integrable. \end{proof}

\begin{lemma}
\label{the:MomentIFS}
For any integrable positive random variable $X$ and for any $p>1$,
\[
 \E X^p
 = p \int_0^\infty H(t^{1/(p-1)}) \, dt,
\]
where $H(t) = \E(X-t)_+$.
\end{lemma}
\begin{proof}
Denote $G(u) = \pr(X>u)$. Because $\pr(X^p > u) = G(u^{1/p})$, a change of variables shows that
\[
 \E X^p
 = \int_0^\infty G(u^{1/p}) \, du
 = p \int_0^\infty G(u) u^\epsilon \, du,
\]
where $\epsilon = p-1$. Further, by writing $u^\epsilon = \int_0^\infty 1(t^{1/\epsilon} < u) \,
dt$ and changing the order of integration, we find that
\[
 p \int_0^\infty G(u) u^\epsilon \, du
 = p \int_0^\infty \left( \int_{t^{1/\epsilon}}^\infty G(u) \, du \right) dt
 = p \int_0^\infty H(t^{1/\epsilon}) \, dt.
\]
\end{proof}

\bibliographystyle{abbrv}
\bibliography{lslReferences}

\end{document}